\documentclass{amsart}
\usepackage{amsmath, amssymb, amsthm, amscd, amsfonts, eucal, hyperref}
\usepackage[all]{xy}
\usepackage{enumerate, color}
\usepackage[T5, T1]{fontenc}
\usepackage{geometry, array, caption, subfig}

\makeatletter
\@namedef{subjclassname@2020}{%
  \textup{2020} Mathematics Subject Classification}
\makeatother

\newcolumntype{C}[1]{>{\centering}m{#1}}

\newtheorem{theorem}{Theorem}[section]
\newtheorem{proposition}[theorem]{Proposition}
\newtheorem{lemma}[theorem]{Lemma}
\newtheorem {corollary}[theorem]{Corollary}
\theoremstyle {definition}
\newtheorem {definition}[theorem]{Definition}
\newtheorem {example}[theorem]{Example}
\newtheorem {question}[theorem]{Question}

\theoremstyle {remark}
\newtheorem{remark}[theorem]{Remark}

\def\ini{\operatorname{in}}

\def\Tor{\operatorname{Tor}}
\def\Im{\operatorname{Im}}
\def\Ker{\operatorname{Ker}}

\def\deg{\operatorname{deg}}

\def\depth{\operatorname{depth}}
\def\reg{\operatorname{reg}}

\newcommand{\fm}{\ensuremath{\mathfrak m}}

\newcommand{\bP}{\ensuremath{\mathbb P}}

\newcommand{\bZ}{\ensuremath{\mathbb Z}}

\begin{document}

\title[Componentwise linearity of projective varieties with almost maximal degree]{Componentwise linearity of projective varieties\\ with almost maximal degree}

\author{\fontencoding{T5}\selectfont \DJ o\`an Trung C\uhorn{}\`\ohorn ng}
\address{\fontencoding{T5}\selectfont \DJ o\`an Trung C\uhorn{}\`\ohorn ng. Institute of Mathematics and the Graduate University of Science and Technology, Vietnam Academy of Science and Technology, 18 Hoang Quoc Viet, 10307 Hanoi, Vietnam.}
\email{dtcuong@math.ac.vn}

\author{Sijong Kwak}
\address{Sijong Kwak. Department of Mathematical Sciences, Korea Advanced Institute of Science and Technology, 373-1 Gusung-dong, Yusung-Gu, Daejeon, Republic of Korea}
\email{sjkwak@kaist.ac.kr}

\thanks{\fontencoding{T5}\selectfont \DJ o\`an Trung C\uhorn{}\`\ohorn ng is funded by Vietnam National Foundation for Science and Technology Development (NAFOSTED) under grant number 101.04-2019.314}
\thanks{Sijong Kwak was supported by the National Research Foundation of Korea (NRF) grant funded by the Korea government (2017R1E1A1A03070765).}

%\date{ }
\begin{abstract}
The degree of a projective subscheme has an upper bound $\deg(X)\leq \binom{e+r}{e}$ in terms of the codimension $e$ and the reduction number $r$. It was proved in  \cite{CK18} that $\deg(X)=\binom{e+r}{e}$ if and only if $X$ is arithmetically Cohen-Macaulay and has an $(r+1)$-linear resolution. Moreover, if the degree of a projective variety $X$ satisfies $\deg(X)=\binom{e+r}{e}-1$, then the Betti table is described with some constraints. In this paper, we build on this work to show that most of such varieties are componentwise linear and the componentwise linearity is particularly suitable for understanding their Betti tables. As an application, the graded Betti numbers of those varieties with componentwise linear resolutions are computed.
\end{abstract}
\maketitle

%%%%%%%%%%%%%%%%%%%%%%%%%%%%%%

\section{Introduction}

Let $X\subset \bP^{n+e}$ be a non-degenerate closed subscheme of dimension $n$ and codimension $e$ over an infinite field $k$. Let $I_X$ be the saturated homogeneous defining ideal and $R_X=k[x_0, \ldots, x_{n+e}]/I_X$ be the homogeneous coordinate ring of $X$. 

Among the important numerical invariants of $X$ there are the degree $\deg(X)$, the Castelnuovo-Mumford regularity and the reduction number $r=r(X)$. The latter is defined as the reduction number of  $R_X$ and we always have $\reg(R_X)\geq r$ by \cite[Proposition 3.2]{NVT87}. The reduction number together with the codimension provides an upper bound for the degree, namely,
%\begin{equation*}
$$\deg(X)\leq \binom{e+r}{e}.$$
%\end{equation*}

\noindent Those projective subschemes attaining the degree upper bound are exactly arithmetically Cohen-Macaulay (abbr. ACM) subschemes with $(r+1)$-linear free resolutions (see \cite[Theorem 1.1]{CK18}). The next to the extremal case has  also been investigated. We say that $X$ is a projective subscheme of almost maximal degree if $\deg(X)=\binom{e+r}{e}-1$. As the main result of \cite{CK18}, the authors gave a description with some constraints of the Betti table of projective varieties of almost maximal degree (see \cite[Theorem 1.3, Theorem 1.4]{CK18} and Proposition \ref{35} of the present paper). For instance, if the projective variety $X$ is of almost maximal degree and is not ACM  then the $(i,j)$-th Betti number of $R_X$, for $i,j>0$, satisfies:

\begin{enumerate}
\item[(a)] If $\reg(R_X)=r$ then
$$\beta_{ij}(R_X)=\begin{cases}
\binom{e+r}{i+r}\binom{r+i-1}{r}+\binom{e}{i-1} &\mbox{ if } j=r, 1\leq i\leq e+1,\\
0&\mbox{ otherwise.}
\end{cases}
$$
\item[(b)] If $\reg(R_X)=r+1$ then 
$$\beta_{i, r}(R_X)-\beta_{i-1,r+1}(R_X)=\binom{e+r}{i+r}\binom{r+i-1}{r}-\binom{e}{i-2}.$$
for $1\leq i\leq e+1$ and $\beta_{ij}(R_X)=0$ for either $i>e+1$ or $j\not=r, r+1$.
\item[(c)] If $\reg(R_X)>r+1$ then
$$\beta_{ij}(R_X)=\begin{cases}
\binom{e+r}{i+r}\binom{i+r-1}{r} &\mbox{ if } j=r, 1\leq i\leq e,\\
\binom{e}{i-1}  &\mbox{ if } j=\reg(R_X), 1\leq i\leq e+1,\\
0 &\mbox{ otherwise.}
\end{cases}$$
\end{enumerate}

\noindent The case (b) with $\reg(R_X)=r+1$ is the only case of which the Betti numbers are not given precisely. It seems to be natural to desire a precise formula for the Betti numbers in this remaining case. In this paper, with this problem in mind, we investigate the structure of the saturated defining ideal of projective varieties of almost maximal degree and show that the componentwise linearity is suitable for our purpose.

\begin{definition}\label{comp.linear}
A homogeneous ideal $I$ is componentwise linear if the ideal $I_{\langle d \rangle}$ generated by degree $d$ homogeneous polynomials in $I$ has a linear minimal free resolution for all $d$.
\end{definition}

It should be noted that the ideal $I_{\langle d \rangle}$ is componentwise linear for all $d\geq \reg(I)$.

Herzog and Hibi introduced componentwise linear ideals in \cite{HH99} as an extension of ideals with linear resolutions. They showed that a Stanley-Reisner ideal $I_\Delta$ is componentwise linear if and only if the Alexander dual $\Delta^*$ is sequentially Cohen-Macaulay, generalizing a well-known theorem of Eagon-Reisner on the equivalence between linearity of $I_\Delta$ and Cohen-Macaulayness of $\Delta^*$. While componentwise linear monomial ideals are studied extensively by many authors (see, for example \cite{HH99, HH11, HHMT08}), componentwise linear prime ideals have not been understood well. 

In the literature, the componentwise linearity of curves has been considered by several authors. For examples, the tetrahedral curves with componentwise linear resolutions are characterized by Francisco-Migliore-Nagel in \cite[Corollary 4.9]{FMN06}. Almost all curves with maximal Hartshorne-Rao module with respect to their degree and genus are componentwise linear in characteristic zero (see \cite[Corollary 6.2]{Nagel03}). A projective subscheme of maximal degree has a linear resolution (see \cite[Theorem 1.1]{CK18}), so it is componentwise linear. By \cite[Theorem 4.1]{CK18}, an ACM  projective subscheme $X$ has an almost maximal degree if and only if $\dim_k(I_X)_r=1$ and the truncated ideal $(I_X)_{\geq r+1}$ has a linear resolution, where $r$ is the reduction number of $R_X$. Consequently $X$ has a componentwise linear resolution. 

For a homogeneous ideal, there is an interesting characterization due to  Aramova-Herzog-Hibi \cite{AHH00} and Nagel-R\"omer \cite{NR15}: an ideal $I$ is componentwise linear if and only if the generic initial ideal  with respect to the degree reverse lexicographic order $\mathrm{Gin}(I)$ is stable and both ideals $I$ and $\mathrm{Gin}(I)$ have the same graded Betti numbers. In the situation of varieties of almost maximal degree, the graded Betti numbers have been computed explicitly, so the characterization suggests possibility to discover the componentwise linearity of these varieties.

Our current research is motivated by the following question of Satoshi Murai.

\begin{question}\label{Murai}\label{11}
Are the saturated defining ideals of projective varieties of almost maximal degree componentwise linear?
\end{question}

\noindent The aim of this paper is to give a complete answer to Question \ref{Murai}. By the discussion above, it suffices to consider non-ACM  projective varieties of almost maximal degree. We will show that most of them have componentwise linear resolutions and at the same time characterize those without componentwise linear resolutions (see Theorem \ref{45}). We actually enlarge our category to include not only varieties of almost maximal degree but all projective subschemes of almost maximal degree and almost maximal arithmetic depth, say, $\depth(R_X)=n=\dim(X)$. 

About the structure of the paper, in Section \ref{Sec3} we characterize projective subschemes of almost maximal degree and almost maximal arithmetic depth in terms of their initial ideal and compute their graded Betti numbers. Componentwise linearity is studied in Section \ref{Sec4} where we obtain two main results (Theorem \ref{45} and Corollary \ref{44}) characterizing componentwise linearity of projective subschemes of almost maximal degree and almost maximal arithmetic depth and computing the graded Betti numbers of those subschemes.

Through this paper, $k$ is an infinite field and a projective variety is a reduced and irreducible projective subscheme. The computation in this paper is established by using Macaulay 2 (cf. \cite{GS}).

%%%%%%%%%%%%%%%%%%%%%%%%%%%%%%%

\medskip
\noindent{\bf Acknowledgments.} The authors thank Satoshi Murai for raising Question \ref{11} which has initially inspired this work. They also thanks \fontencoding{T5}\selectfont Nguy\~\ecircumflex n \DJ\abreve ng H\d\ohorn p for helpful discussion on Macaulay 2, especially for his help to find Example \ref{48}. Part of this work has been done during the visit of Sijong Kwak to the Institute of Mathematics, Vietnam Academy of Science and Technology (IM-VAST). He thanks the IM-VAST and the Simons Foundation Targeted Grant for the IM-VAST (No. 558672) for support and hospitality. \DJ o\`an Trung C\uhorn{}\`\ohorn ng thanks the Vietnam Institute for Advanced Study in Mathematics (VIASM) for support and hospitality during his visit to the VIASM in 2020.

%%%%%%%%%%%%%%%%%%%%%%%%%%%%%%

\section{Projective subschemes of almost maximal degree and almost maximal arithmetic depth} \label{Sec3}

Let $X\subset \bP^{n+e}$ be a non-degenerate projective subscheme of dimension $n$ and codimension $e$ over the field $k$. Denote by $S_i$ the polynomial ring $k[x_i, \ldots, x_{n+e}]$ for $i=0, 1, \ldots, n+e$. Let $I_X\subset S_0$ be the saturated homogeneous defining ideal of $X$ and $R_X=S_0/I_X$ be the homogeneous coordinate ring. 

Let $S$ be a $k$-subalgebra of $R_X$ generated by linear forms such that $S\hookrightarrow R_X$ is a Noether normalization, i.e., $S$ is a polynomial $k$-algebra and $R_X$ is a finitely generated $S$-module. The reduction number of $R_X$ with respect to $S$ is the supremum of degree of all homogeneous minimal generators of $R_X$ as an $S$-module, denoted $r_S(R_X)$. The reduction number of $R_X$ is the least $r_S(R_X)$ where $S$ runs over all Noether normalization (see \cite{Vas96}). This number, say $r$, is also called the reduction number of $X$.

One might change the variables by a linear transformation such that $S_e=k[x_e, \dots, x_{n+e}]\rightarrow R_X$ is a Noether normalization of $R_X$ whose reduction number is exactly $r$. Then we have the upper bounds for degree $\deg(X)=\deg(R_X)\leq \binom{e+r}{e}$ (see \cite[Theorem 1.1]{CK18}). On the other hand, we always have the upper bound of the arithmetic depth of $X$, namely, $\depth(R_X)\leq n+1$. 

We say that $X$ is a subscheme of maximal degree if $\deg(X)=\binom{e+r}{e}$ and that $X$ is of {\it almost maximal degree} if $\deg(X)=\binom{e+r}{e}-1$. If $\depth(R_X)=n+1$ then $X$ is ACM. If $\depth(R_X)=n$ then we say that $X$ has an {\it almost maximal arithmetic depth}. A projective variety of maximal degree is always ACM  (see, for example, \cite[Theorem 1.1]{CK18}). If a projective variety $X$ is of almost maximal degree then $\depth(R_X)\geq n$ (see  \cite[Theorem 5.1]{CK18}). In the sequel, instead of considering only varieties, we enlarge our category to include all projective subschemes with almost maximal degree and almost maximal arithmetic depth. We firstly extend some results on initial ideals and graded Betti numbers of projective varieties of almost maximal degree in \cite{CK18} to these projective subschemes.

\begin{theorem}[Initial ideal]\label{31}
Let $X\subset \bP^{n+e}$ be a non-degenerate closed subscheme of dimension $n$, codimension $e$ and reduction number $r$. Assume that $S=S_e=k[x_e, \ldots, x_{n+e}]$ is a Noether normalization of $R_X$ with reduction number $r_S(R_X)=r$. We fix the degree reverse lexicographic order on the monomials of $S_0$. The following statements are equivalent:
\begin{enumerate}[(a)]
\item $\deg(X)=\binom{e+r}{e}-1$ and $\depth(R_X)=n$;
\item $\ini(I_X)=(x_0, \ldots, x_{e-1})^{r+1}+(uv)$, where $u\in k[x_0, \ldots, x_{e-1}]$ is a monomial of degree $r$ and $v\in k[x_e, \ldots, x_{n+e}]$ is a monomial of positive degree.
\end{enumerate} 
If it is the case, then $\reg(R_X)=\reg(S_0/\ini(I_X))=\deg(uv)-1$.
\end{theorem}
\begin{proof}

\noindent $(b)\Rightarrow (a)$: The degree conclusion follows from the comparison
$$\deg(R_X)=\deg(S_0/\ini(I_X))=\binom{e+r}{e}-1.$$

In order to prove the depth conclusion, let $J=(x_0, \ldots, x_{e-1})$. We have a short exact sequence
$$0\rightarrow S_0/J^{r+1}\stackrel{*uv}{\longrightarrow} S_0/J^{r+1}\rightarrow S_0/\ini(I_X) \rightarrow 0.$$
Since $S_0/J^{r+1}$ is Cohen-Macaulay of dimension $n+1$, we obtain
$$\depth(S_0/\ini(I_X))\geq n.$$

Now the cancellation principle induces the inequalities of Betti numbers $\beta^{S_0}_{ij}(S_0/I_X)\leq \beta^{S_0}_{ij}(S_0/\ini(I_X))$ for any $i, j$ (see \cite[Corollary 1.21]{Green98} or \cite[Section 3.3]{HH11}) which leads to a comparison of the projective dimension
$$\mathrm{proj.dim}_{S_0}(S_0/I_X)\leq \mathrm{proj.dim}_{S_0}(S_0/\ini(I_X)).$$
The Auslander-Buchsbaum formula then implies that $\depth(S_0/I_X)\geq \depth(S_0/\ini(I_X))\geq n$. As $S_0/I_X$ is not Cohen-Macaulay (see \cite[Proposition 2.1]{BG01}), $\depth(S_0/I_X)=n$.

\bigskip
\noindent $(a)\Rightarrow (b)$: Suppose $\deg(X)=\binom{e+r}{e}-1$ and $\depth(R_X)=n$.

The $S_e$-module $R_X$ has a minimal set of generators containing all monomials in $x_0, \ldots, x_{e-1}$ which are not contained in the ideal $\ini(I_X)+(x_e, \ldots, x_{n+e})$ (see \cite[Lemma 2.4]{CK18}). Denote this set by $B_0$ and its cardinality by $\mu_{S_e}(R_X)$. Let $r_1$ be the maximal degree of monomials in $B_0$. The reduction number is the maximal degree of homogeneous minimal generators of $R_X$ over $S_e$ (see \cite{Vas96}), so $r_1\geq r$ and $R_X$, as an $S_e$-module, is also generated by all monomials in $x_0, \ldots, x_{e-1}$ of degree at most $r$. This set of generators thus contains $B_0$ and consequently we obtain $r_1=r$. This leads to the comparison
$$ \mu_{S_e}(R_X)\leq \binom{e+r}{e}.$$
Since $X$ is not ACM, we have
$$ \binom{e+r}{e}-1=\deg(X)< \mu_{S_e}(R_X).$$
Hence $\mu_{S_e}(R_X)=\binom{e+r}{e}$. This shows that actually we have
$$B_0=\{x_0^{n_0}\ldots x_{e-1}^{n_{e-1}}: n_0+\cdots+n_{e-1}\leq r\}.$$
Therefore the initial ideal $\ini(I_X)$ is minimally generated by a disjoint union of the set $T_{r+1}$ of all monomials in $x_0, \ldots, x_{e-1}$ of degree $r+1$ and a set $M$ of some monomials in $S_0$ which is divisible by some $x_i$ for $i\geq e$. Write $M=\{u_1v_1, \ldots, u_sv_s\}$ where $u_i \in k[x_0, \ldots, x_{e-1}]$ with $0<\deg(u_i)\leq r$ and $v_i\in k[x_e, \ldots, x_{n+e}]$ with $\deg(v_i)>0$. This shows particularly that $(x_0,\ldots,x_{e-1})$ is the unique associated prime ideal of $\ini(I_X)$ and that the corresponding associated  primary component is $(u_1,\ldots,u_s, T_{r+1})$. Hence
$$\deg(S_0/\ini(I_X))=\deg(S_0/(u_1, \ldots, u_s, T_{r+1})).$$ 
We also have
$$\deg(S_0/\ini(I_X))=\deg(S_0/I_X)=\binom{e+r}{e}-1.$$
So $\deg(S_0/(u_1, \ldots, u_s, T_{r+1}))=\binom{e+r}{e}-1$. Since $S_0/(u_1, \ldots, u_s, T_{r+1})$ is Cohen-Macaulay, the degree of $S_0/(u_1, \ldots, u_s, T_{r+1})$ is the number of monomials in $T_0\cup T_1\cup \cdots \cup T_r$ which are not in $\{u_1, \ldots, u_s\}$. Since $|T_0\cup T_1\cup \cdots \cup T_r|=\binom{e+r}{e}$, it induces the identity $u_1=u_2=\cdots=u_s=u$ with $\deg(u)=r$. 

\medskip
In order to obtain the desired description of the initial ideal of $I_X$, it is essential to study the minimal free resolution of $R_X$ as an $S_e$-module.

To start with, we prove that the equivalence classes in $R_X$ of the monomials in $B_0\setminus \{u\}$ are $S_e$-linearly independent. Let us  denote the monomials in $B_0\setminus \{u\}$ by $u_1, \ldots, u_d$, where $d=\binom{e+r}{r}-1=\deg(R_X)$. Assume that 
$$f_1\bar u_1+\cdots+f_{d}\bar u_d=0,$$
for some polynomials $f_1, \ldots, f_{d}\in S_e$ which are not identically zero. Let $f=f_1u_1+\cdots+f_{d}u_{d}\in I$. We can assume in addition that $f_1, \ldots, f_{d}$ are homogeneous polynomials such that $f$ is also homogeneous.  Obviously $u_1, \ldots, u_{d}\in S_0$ are linearly independent over $S_e$, so $f\not=0$. Write $\ini(f)=\lambda m_1m_2$, where $\lambda\in k^\times$, $m_1\in B_0\setminus \{u\}$ and $m_2\in S_e$. Then this contradicts to the fact that $\ini(f)$ lies in $\ini(I)$ which is minimally generated over $S_e$ by $T_{r+1}\cup \{uv_1, \ldots, uv_s\}$. This shows that $\overline u_1, \ldots, \overline u_{d}$ are $S_e$-linearly independent.

\medskip
Now the Auslander-Buchsbaum formula induces 
$$\mathrm{proj.dim}_{S_e}(R_X)=\depth(S_e)-\depth(R_X)=1.$$
Since $\mu_{S_e}(R_X)=d+1$, $\deg(R_X)=d$, $R_X$ as an $S_e$-module has a minimal graded $S_e$-free resolution
$$0\longleftarrow R_X\stackrel{\phi}{\longleftarrow}F_0= \bigoplus_{i=1}^{d+1}S_ee_i\stackrel{\psi}{\longleftarrow} F_1=S_eg\longleftarrow 0,$$
where $\{e_1, \ldots, e_{d+1}\}$ and $\{g\}$ are bases of the free $S_e$-modules $F_0, F_1$ and 
$$\phi(e_i)=
\begin{cases}
\overline u_i&\mbox{ for } i=1, \ldots, d,\\
\overline u  &\mbox{ for } i=d+1,
\end{cases}$$
$$\psi(g)=(h_ie_i)_{i=1, \ldots, d+1}=:\omega\in F_0,$$
for some homogeneous polynomials $h_1, \ldots, h_{d+1}\in S_e$. The homomorphism $\phi$ induces an isomorphism
$$\overline \phi: \bigoplus_{i=1}^{d+1}S_ee_i\big/(\sum_{i=1}^{d+1}h_ie_i)\stackrel{\simeq}{\longrightarrow} R_X.$$
The polynomial $h_{d+1}$ is particularly non-zero as $\overline u_1, \ldots, \overline u_d$ are $S_e$-linearly independent.

Recall that $T_{r+1}\cup\{uv_1, \ldots, uv_s\}$ is a minimal set of generators of $\ini(I_X)$. Let $g_i$ be the polynomial in the reduced Gr\"obner basis of $I_X$ with $\ini(g_i)=uv_i$, for $i=1, \ldots, s$. Then no trailing monomials of $g_i$ lie in the initial ideal $\ini(I_X)$ and we can write
$$g_i=uq_i+\sum_{j=1}^{d}u_jq_{ij},$$
with some homogeneous polynomials $q_i, q_{ij}\in S_e$ and $\ini(q_i)=v_i$. In the $S_e$-module $F_0=\bigoplus_{j=1}^{d+1}S_ee_j$ we consider the elements
$$\omega_i=q_ie_{d+1}+\sum_{j=1}^dq_{ij}e_j, i=1, \ldots, s.$$
We have $\phi(\omega_i)=\overline g_i=0$, hence $\omega_i\in \Ker(\phi)=\Im(\psi)=(\omega)$ and we can write $\omega_i=a_i\omega$ for some polynomial $a_i\in S_e$. We denote $h=\sum_{i=1}^{d}h_iu_i+h_{d+1}u\in I_X$, then obviously $g_i=\phi(\omega_i)=a_i\phi(\omega)=a_ih$. Since $g_i$ is in the reduced  Gr\"obner basis of $I_X$ and $h\in I_X$, this is possible only if $a_i$ is a non-zero constant polynomial, i.e., $a_i\in k^\times$. This deduces that $s=1$ and
$$\ini(I_X)=(T_{r+1})+(uv_1).$$

\bigskip
It remains to prove the last conclusion on Castelnuovo-Mumford regularity. Suppose $X$ satisfies $(a)$ and $(b)$. Recall that $R_X$ has a minimal graded free $S_e$-resolution
$$0\longleftarrow R_X\stackrel{\phi}{\longleftarrow}F_0= \bigoplus_{i=1}^{d+1}S_ee_i\stackrel{\psi}{\longleftarrow} F_1=S_eg\longleftarrow 0,$$
where $S_ee_i\simeq S_e[-\deg u_i]$ for $i=1,\ldots, d$, $S_ee_{d+1}\simeq S_e[-\deg u]$ and $S_eg\simeq S_e[-\deg h]$. Moreover, $\deg(h)=\deg(g_1)=\deg(uv_1)\geq r+1$. Hence 
\[\begin{aligned}
\reg(R_X)
&=\max\{\deg(u_1), \ldots, \deg(u_d), \deg(u), \deg(h)-1\}\\
&=\deg(h)-1=\deg(uv)-1=\reg(S_0/\ini(I_X)).
\end{aligned}\]
\end{proof}

Let $M$ be a finitely generated graded module over a polynomial ring $S$.  The $(i, j)$-th graded Betti number of $M$ is $\beta^{S_0}_{ij}=\dim_k\Tor_i^{S_0}(M, k)_{i+j}$. In the next we are going to compute the Betti numbers of projective subschemes satisfying the equivalent conditions in Theorem \ref{31}. Let us consider first some examples of such subschemes.

\begin{example}\label{32} Belows are some examples of projective subschemes of almost maximal degree and almost maximal arithmetic depth.

\begin{enumerate}[(i)]
\item \cite[Theorem 5.1]{CK18} Projective varieties of almost maximal degree have almost maximal arithmetic depth.
\item \cite[Lemma 5.3]{CK18} Let $S_0=k[x_0, \ldots, x_{n+e}]$ and $J=(x_0, \ldots, x_{e-1})$. Let $u\in J$ be a monomial of degree $r$ and $v$ be a non-constant monomial in $x_e, \ldots, x_{n+e}$. Put $I=(uv)+J^{r+1}$. Then $I$ defines a closed subscheme in $\bP^{n+e}$ which satisfies the equivalent conditions in Theorem \ref{31}. 

The Betti numbers of $I$ are as follows:
\begin{enumerate}[(a)]
\item If $\deg(uv)=r+1$ then
$$\beta_{ij}^{S_0}(I)=\begin{cases}
\binom{e+r}{i+r+1}\binom{r+i}{r}+\binom{e}{i} &\mbox{ if } 0\leq i<e, j=r+1,\\
0 &\mbox{ otherwise.}
\end{cases}
$$
\item If $\deg(uv)>r+1$ then
$$\beta_{ij}^{S_0}(I)=\begin{cases}
\binom{e+r}{i+r+1}\binom{r+i}{r}&\mbox{ if } 0\leq i<e, j=r+1,\\
\binom{e}{i} &\mbox{ if } 0\leq i\leq e, j=\deg(uv),\\
0 &\mbox{ otherwise.}
\end{cases}
$$
\end{enumerate}
\end{enumerate}
\end{example}

\begin{corollary}\label{33}
Let $X\subset \bP^{n+e}$ be a non-degenerate closed subscheme of dimension $n$, codimension $e$ and reduction number $r$. Let $R_X=S_0/I_X$ be the homogeneous coordinate ring of $X$. Suppose $S=S_e\rightarrow R_X$ is a Noether normalization with reduction number $r_S(R_X)=r$. The following statements are equivalent:
\begin{enumerate}[(a)]
\item  $\deg(X)=\binom{e+r}{e}-1$ and $\depth(R_X)=n$.
\item $R_X$, as an $S$-module, has the graded Betti numbers
$$\beta_{i,j}^S(R_X)=\begin{cases}
\binom{e+j-1}{j} &\mbox{ if } i=0, 0\leq j\leq r,\\
1 &\mbox{ if } i=1, j=\reg(R_X),\\
0 &\mbox{ if } i=1, j\not=\reg(R_X) \text{ or } i>1.
\end{cases}$$
\item $R_X$, as an $S$-module, has the Betti numbers
$$\beta_i^S(R_X)=\begin{cases}
\binom{e+r}{r} &\mbox{ if } i=0,\\
1 &\mbox{ if } i=1,\\
0 &\mbox{ if } i>1.
\end{cases}$$
\end{enumerate}
\end{corollary}
\begin{proof}
The implications $(b)\Rightarrow (c)\Rightarrow (a)$ are obvious. We are going to show $(a)\Rightarrow (b)$.

Assume that $X$ has $\deg(X)=\binom{e+r}{e}-1$ and $\depth(R_X)=n$. By the proof for $(b)\Rightarrow(a)$, Theorem \ref{31}, we have
$$\beta_{i,j}^{S}(R_X)=\begin{cases}
1 &\mbox{ if } i=1, j=\reg(R_X),\\
0 &\mbox{ if } i>1\text{ or } i=1, j\not=\reg(R_X) .
\end{cases}$$
Now we compute $\beta_{0,j}^{S}(R_X)$. We have seen in the proof of Theorem \ref{31} that $R_X$ is minimally generated over $S$ by all monomials in $x_0, \ldots, x_{e-1}$ of degree from $0$ to $r$. Consequently, we have
$$\sum_{j=0}^r\beta_{0,j}^S(R_X)=\binom{e+r}{r}.$$
On the other hand, $\beta_{0,j}^S(R_X)$ is bounded above by the number of monomials in $x_0, \ldots, x_{e-1}$ of degree $j$, i.e.,
$$\beta_{0,j}^S(R_X)\leq \binom{e+j-1}{j}.$$
This implies that
$$\beta_{0,j}^S(R_X)=\binom{e+j-1}{j},$$
for all $j=0, 1, \ldots, r$.
\end{proof}

As an immediate consequence, we get information on Hilbert polynomial and arithmetic genus.

\begin{corollary}\label{34}
Let $X\subset \bP^{n+e}$ be a non-degenerate closed subscheme of dimension $n$, codimension $e$ and reduction number $r$. Let $R_X=S_0/I_X$ be the homogeneous coordinate ring of $X$. Suppose that $\deg(X)=\binom{e+r}{e}-1$ and $\depth(R_X)=n$. The Hilbert polynomial of $X$ is
$$P_{R_X}(T)=\sum_{j=0}^r\binom{e-1+j}{e-1}\binom{T+n-j}{n}-\binom{T-\reg(R_X)-1+n}{n}.$$
The arithmetic genus of $X$ is
$$g(X)=(-1)^n(P_{R_X}(0)-1)=\sum_{j=n+1}^r\binom{e-1+j}{e-1}\binom{j-1}{n}-\binom{\reg(R_X)}{n}+(-1)^{n+1}.$$
In particular, if $r\leq \dim(X)$ then 
$$g(X)=-\binom{\reg(R_X)}{n}+(-1)^{n+1}.$$
\end{corollary}

So over the Noether normalization $S=S_e$, the Betti table of $X$ is described precisely. From this table we can mostly recover the Betti table of $X$ over the ring $S_0$. 

The following description of Betti tables was shown in \cite[Theorem 1.4]{CK18} for varieties. Now we generalize it for projective subschemes and actually give a characterization of projective subschemes with almost maximal degree and almost maximal arithmetic depth.

\begin{proposition}[Betti table]\label{35}
Let $X\subset \bP^{n+e}$ be a non-degenerate closed subscheme of dimension $n$, codimension $e$ and reduction number $r$. Let $I_X\subset S_0$ be the saturated defining ideal of $X$ and $R_X=S_0/I_X$. Then $\deg(X)=\binom{e+r}{e}-1$ and $\depth(R_X)=n$ if and only if the Betti table (over $S_0$) of $R_X$ has one of the following shapes (in the following tables, we write only rows with some possibly non-zero entries):

\begin{enumerate}
\item[(a)] $\reg(R_X)=r$:

\begin{figure}[!htb]
\begin{tabular}{>{\centering}m{1.5cm}|>{\centering}m{1cm} >{\centering}m{1cm}>{\centering}m{1cm} >{\centering}m{1cm} >{\centering}m{1cm} >{\centering}m{1cm} c}
	&	$0$	&		$1$&	$\ldots$ &	$i$&	$\ldots$ &$e$&	$e+1$\\
\hline
$0$	&	$1$	&		--&		$\ldots$	&	--&		$\ldots$	&--&		--\\
$r$	&	--&	 	$\beta_{1,r}$		&	$\ldots$	&	$\beta_{i,r}$		&	$\ldots$	&$\beta_{e,r}$&		$1$\\
\end{tabular}
\end{figure}

\noindent where for $1\leq i\leq e+1$,
$$\beta_{i,r}=\binom{e+r}{i+r}\binom{r+i-1}{r}+\binom{e}{i-1}.$$

%\newpage
\item[(b)] $\reg(R_X)=r+1$:

\begin{figure}[!htb]
\begin{tabular}{>{\centering}m{1.5cm}|>{\centering}m{1cm} >{\centering}m{1cm}>{\centering}m{1cm} >{\centering}m{1cm} >{\centering}m{1cm} >{\centering}m{1cm} c}
	&	$0$	&		$1$&	$\ldots$ &	$i$&	$\ldots$ &$e$&	$e+1$\\
\hline
$0$	&	$1$	&		--&		$\ldots$	&	--&		$\ldots$	&--&		--\\
$r$	&	--&	 	$\beta_{1,r}$		&	$\ldots$	&	$\beta_{i,r}$		&	$\ldots$	&$\beta_{e,r}$&		--\\
$r+1$	&	--&	 	$\beta_{1,r+1}$		&	$\ldots$	&	$\beta_{i,r+1}$		&	$\ldots$	&$\beta_{e,r+1}$&		$1$\\
\end{tabular}
\end{figure}

\noindent where for $1\leq i\leq e+1$,
$$\beta_{i, r}-\beta_{i-1,r+1}=\binom{e+r}{i+r}\binom{r+i-1}{r}-\binom{e}{i-2},$$
and
$$\beta_{i, r}\leq \binom{e+r}{i+r}\binom{r+i-1}{r},\ \ \ \ \beta_{i-1,r+1}\leq \binom{e}{i-2},$$

\item[(c)] $\reg(R_X)>r+1$:
%\newpage

\begin{figure}[!htb]
\begin{tabular}{>{\centering}m{1.5cm}|>{\centering}m{1cm} >{\centering}m{1cm}>{\centering}m{1cm} >{\centering}m{1cm} >{\centering}m{1cm} >{\centering}m{1cm} c}
	&	$0$	&		$1$&	$\ldots$ &	$i$&	$\ldots$ &$e$&	$e+1$\\
\hline
$0$	&	$1$	&		--&		$\ldots$	&--&		$\ldots$	&	--&		--\\
$r$	&	--&	 	$\beta_{1r}$		&	$\ldots$	&$\beta_{ir}$		&	$\ldots$	&	$\beta_{e,r}$&		--\\
$\reg(R_X)$	&	--&	 	$\binom{e}{0}$		&	$\ldots$	&$\binom{e}{i-1}$		&	$\ldots$	&	$\binom{e}{e-1}$&		$\binom{e}{e}$\\
\end{tabular}
\end{figure}

\noindent where for $1\leq i\leq e+1$,
$$\beta_{i,r}=\binom{e+r}{i+r}\binom{i+r-1}{r},$$
$$\beta_{i,\reg(R_X)}=\binom{e}{i-1}.$$
\end{enumerate}
\end{proposition}
\begin{proof}
Suppose $S=S_e\rightarrow R_X$ is a Noether normalization of $R_X$ with reduction number $r_S(R_X)=r$. For the proof we make use of the relation between Betti numbers of $R_X$ over $S_0$ and over $S$. For each $m\in \bZ$, denote
$$\chi_m^{S_0}(M)=\sum_{j=0}^m(-1)^j\beta_{m-j, j}^{S_0}(M).$$
\medskip

\noindent Necessary condition: The proof of the necessary condition is similar to the proof of Theorem 1.4 in \cite{CK18} which treats varieties instead of subschemes. Suppose $\deg(X)=\binom{e+r}{e}-1$ and $\depth(R_X)=n$. Fix the degree reverse lexicographic order on monomials of $S_0$. Due to Theorem \ref{31}, the initial ideal of $I_X$ has the explicit description
$$\ini(I_X)=(x_0,\ldots, x_{e-1})^{r+1}+(uv),$$
for some monomials $u$ in $x_0, \ldots, x_{e-1}$ of degree $r$ and $v$ in $x_e, \ldots, x_{n+e}$ of positive degree.

The cancellation principle (see \cite[Corollary 1.21]{Green98} or \cite[Section 3.3]{HH11}) gives rise to the comparison
$$0\leq \beta_{ij}^{S_0}(R_X)\leq \beta_{ij}^{S_0}(S_0/\ini(I_X)).$$
Now using the computation in Example (ii) of Lemma \ref{32} we obtain
$$\beta_{ij}^{S_0}(R_X)=0,$$
for all $(i, j)\not\in\{(0,0), (1, r), \ldots, (e,r), (1, \reg(R_X)), \ldots, (e+1, \reg(R_X))\}$.

If $\reg(R_X)=r$ or $\reg(R_X)\geq r+2$ then we get
$$\beta_{ij}^{S_0}(R_X)=\beta_{ij}^{S_0}(S_0/\ini(I_X)),$$
for all $i, j$, the conclusion then follows.

Suppose $\reg(R_X)=r+1$, then 
$$\chi_m^{S_0}(R_X)=(-1)^r\beta_{m-r, r}^{S_0}(R_X)+(-1)^{r+1}\beta_{m-r-1, r+1}^{S_0}(R_X).$$
For the Betti numbers of $R_X$ over the Noether normalization $S$, it follows from Corollary \ref{33} that
$$\chi_m^{S}(R_X)=\begin{cases}
(-1)^m\binom{e+m-1}{e-1} &\mbox{ if } 0\leq m\leq r,\\
1 &\mbox{ if } m=\reg(R_X)+1,\\
0 &\mbox{ otherwise.}
\end{cases}$$
Following \cite[Corollary 2.7(a)]{CK18}, we have
$$\chi_m^{S_0}(R_X)=\sum_{j=0}^e\binom{e}{j}\chi_{m-j}^{S_e}(R_X)=(-1)^r\binom{e+r}{m}\binom{m-1}{r}+(-1)^{r+1}\binom{e}{m-r-2}.$$
It proves the conclusion in this case.

\medskip
\noindent Sufficient condition: From the Betti table, $R_X$ has projective dimension $e+1$ over $S_0$. The Auslander-Buchsbaum formula then gives us $\depth(R_X)=n$. Furthermore, 
we also have
$$\chi_m^{S_0}(R_X)=\sum_{j=0}^e\binom{e}{j}\chi_{m-j}^{S_e}(R_X)=(-1)^r\binom{e+r}{m}\binom{m-1}{r}+(-1)^{\reg(R_X)-r}\binom{e}{m-1-\reg(R_X)}.$$
Now using again the relation between the Betti numbers over $S_0$ and $S$ in \cite[Corollary 2.7(a)]{CK18}, we obtain
$$\chi_m^{S}(R_X)=\begin{cases}
(-1)^m\binom{e+m-1}{e-1} &\mbox{ if } 0\leq m\leq r,\\
1 &\mbox{ if } m=\reg(R_X)+1,\\
0 &\mbox{ otherwise.}
\end{cases}$$
Hence
$$\deg(R_X)=\sum_{m\geq 0}(-1)^m\chi_m^S(R_X)=\sum_{m=0}^r\binom{e+m-1}{e-1}-1=\binom{e+r}{e}-1.$$
Therefore $X$ has an almost maximal degree.
\end{proof}

%%%%%%%%%%%%%%%%%%%%%%%%%%%%%%

\section{Componentwise linearity}\label{Sec4}

For a homogeneous ideal $I\subset S_0$, we denote by $I_{\langle d \rangle}$ the ideal generated by all homogeneous polynomials of degree $d$ in $I$. Following Herzog-Hibi \cite{HH99}, we say that $I$ is componentwise linear if for each $d>0$, the ideal $I_{\langle d \rangle}$ has a linear minimal free resolution. There are several characterizations of componentwise linear ideals, mostly by the equality between the Betti numbers of the ideal and its initial ideal with respect to certain monomial orders (see  \cite{AHH00, NR15}). 

The main aim of this section is to answer Question \ref{Murai} by showing that most of the projective subschemes of almost maximal degree and almost maximal arithmetic depth have componentwise linear resolution. 

The following simple lemma is very useful in the sequel.

\begin{lemma}\label{41}
Let $I\subset S_0$ be a homogeneous ideal with a linear resolution. Then so is $(x_0, \ldots, x_{e+n})I$.
\end{lemma}
\begin{proof}
The ideal $I$ has a set of generators consisting of homogeneous polynomials of the same degree $d=\reg(I)$. Let $t$ be the minimal number of generators of $I$, then
$I/\fm I\simeq k[-d]^t$, where $\fm=(x_0, \ldots, x_{e+n})\subset S_0$. From the short exact sequence
$$0\rightarrow k[-d]^t\rightarrow S_0/\fm I\rightarrow S_0/I\rightarrow 0,$$
we obtain $\reg(S_0/\fm I)\leq \max\{\reg(S_0/I), d\}=d$. On the other hand, $\fm I$ is generated by degree $d+1$ homogeneous elements, so $\reg(S_0/\fm I)=d$ and $\fm I$ has a $(d+1)$-linear resolution (see \cite[Theorem 1.2]{EG84}).
\end{proof}

\begin{corollary}\label{42}
Let $u$ be a monomial in $x_0, \ldots, x_{e-1}$ of degree $r$ and let $v$ be a monomial in $x_e, \ldots, x_{n+e}$ of positive degree. Put $I=(x_0, \ldots, x_{e-1})^{r+1}+(uv)$. Then $I$ has a componentwise linear minimal free resolution.
\end{corollary}
\begin{proof}
By Theorem \ref{31}, $S_0/I$ is of almost maximal degree and almost maximal arithmetic depth. Furthermore, one has $\reg(S_0/I)=\deg(uv)-1$.

If $\deg(v)=1$ then $I$ is generated by homogeneous polynomials of degree $r+1$. Due to Proposition \ref{35}(a), $I$ has a linear resolution.

Let us  assume $\deg(v)>1$. It is clear that $I_{\langle r+1\rangle}=(x_0, \ldots, x_{e-1})^{r+1}$ has an $(r+1)$-linear resolution. For $1\leq s<\deg(v)$, the ideal  $I_{\langle r+s\rangle}=\fm^{s-1}(x_0, \ldots, x_{e-1})^{r+1}$ has a linear resolution due to Lemma \ref{41}. For $s\geq \deg(v)$, we have $r+s\geq \deg(uv)=\reg(I)$ and therefore $I_{\langle r+s\rangle}$ has a linear resolution (see \cite[Proposition 1.1]{EG84}).
\end{proof}

\begin{proposition}\label{43}
Let $I\subset S_0$ be a homogeneous ideal such that $R=S_0/I$ has dimension $n+1$ and reduction number $r$. Suppose $\deg(R)=\binom{e+r}{e}-1$ and $\depth(R)=n$.
\begin{itemize}
\item[(a)] $I_{\langle r+1\rangle}$ has a linear resolution if and only if either
\begin{itemize}
\item[(i)] $\reg(R)=r$, or
\item[(ii)] $\reg(R)=r+1$ and $\beta_{1,r+1}^{S_0}(R)=1$, or
\item[(iii)] $\reg(R)\geq r+2$.
\end{itemize}

When $(ii)$ or $(iii)$ is the case, $S_0/I_{\langle r+1\rangle}$ is Cohen-Macaulay.

\item[(b)] $I_{\langle r+1\rangle}$ does not have a linear resolution if and only if $\reg(R)=r+1$, $\beta_{1, r+1}^{S_0}(R)=0$. 

When it is the case, $\beta_{1,r+1}^{S_0}(S_0/\ini(I_{\langle r+1\rangle}))=1$ and  $I$ is generated by homogeneous polynomials of degree $r+1$.
\end{itemize}
\end{proposition}
\begin{proof}
If $\reg(R)=r$ then $I_{\langle r+1\rangle}$ has an $(r+1)$-linear resolution by Proposition \ref{35}(a). 

Let us assume that $\reg(R)>r$. Changing the variables by a suitable linear transformation, we may assume that $S=S_e=k[x_e, \ldots, x_{n+e}]$ is a Noether normalization of $R$ with reduction number $r_S(R)=r$. We fix the degree reverse lexicographic order on the monomials of $S_0$. By Theorem \ref{31}, the initial ideal of $I$ has a simple form
$$\ini(I)=(x_0, \ldots, x_{e-1})^{r+1}+(uv),$$
where $u$ is a monomial in $x_0, \ldots, x_{e-1}$ of degree $r$ and $v$ is a monomial in $x_e, \ldots, x_{n+e}$ of degree at least $2$. Let $g_1, \ldots, g_t, g$ be a reduced Gr\"obner basis of $I$ such that $\{\ini(g_1), \ldots, \ini(g_t)\}$ is a minimal set of generators of $(x_0, \ldots,x_{e-1})^{r+1}$ and $\ini(g)=uv$. In particular, $I_{\langle r+1\rangle}=(g_1, \ldots, g_t)$. 

Now we prove the equivalence in (b). Suppose $I_{\langle r+1\rangle}$ does not have a linear resolution. We have
$$(x_0, \ldots, x_{e-1})^{r+1}\subseteq \ini(I_{\langle r+1\rangle})\subseteq \ini(I)=(x_0, \ldots, x_{e-1})^{r+1}+(uv),$$
with $\deg(uv)\geq r+2$. If $\ini(I_{\langle r+1\rangle})=(x_0, \ldots, x_{e-1})^{r+1}$ then $\ini(I_{\langle r+1\rangle})$ is $(r+1)$-linear, so $I_{\langle r+1\rangle}$ is $(r+1)$-linear, a contradiction. Then $\ini(I_{\langle r+1\rangle})$ must have a minimal generator of degree at least $r+2$. Following \cite[Lemma 2.2]{NR15}, it induces $\beta_{0,r+2}^{S_0}(\ini(I_{\langle r+1\rangle}))\not=0$. This occurs only if $\deg(uv)=r+2$ and
$$\ini(I_{\langle r+1\rangle})=(x_0, \ldots, x_{e-1})^{r+1}+(uv)=\ini(I).$$
Thus $\beta_{0,r+2}^{S_0}(\ini(I_{\langle r+1\rangle}))=\beta_{0,r+2}^{S_0}(\ini(I))=1$.

Put $\fm=(x_0, \ldots, x_{e+n})$. We have
$$\begin{aligned}
\beta_{0,r+2}^{S_0}(\ini(I))
& =\dim_k\ini(I)_{r+2}-\dim_k(\fm\ini(I))_{r+2}\\
& =\dim_k\ini(I)_{r+2}-\dim_k(\fm\ini(I_{\langle r+1\rangle}))_{r+2}\\
&=\dim_k(I_{r+2})-\big( \dim_k(\ini(I_{\langle r+1\rangle})_{r+2})-\beta_{0,r+2}^{S_0}(\ini(I_{\langle r+1\rangle}))\big).
\end{aligned}$$
Hence 
$$\beta_{0,r+2}^{S_0}(\ini(I))-\beta_{0,r+2}^{S_0}(\ini(I_{\langle r+1\rangle}))=\dim_k(I_{r+2})-\dim_k(\ini(I_{\langle r+1\rangle})_{r+2}).$$
On the other hand, we have
$$\begin{aligned}
\beta_{0,r+2}^{S_0}(I)
&=\dim_k(I_{r+2})-\dim_k((\fm I)_{r+2})\\
&=\dim_k(I_{r+2})-\dim_k((\fm I_{\langle r+1\rangle})_{r+2})\\
&=\dim_k(I_{r+2})-\dim_k((I_{\langle r+1\rangle})_{r+2})\\
&=\dim_k(I_{r+2})-\dim_k(\ini(I_{\langle r+1\rangle})_{r+2}).
\end{aligned}$$
This deduces that $\beta_{0,r+2}^{S_0}(I)=\beta_{0,r+2}^{S_0}(\ini(I))-\beta_{0,r+2}^{S_0}(\ini(I_{\langle r+1\rangle}))=0$. 

Conversely, if $\beta_{0,r+2}^{S_0}(I)=0$ then by Proposition \ref{35}(b),  $I_{\langle r+1\rangle}=I$ which does not have a linear resolution. This proves (b). It also proves  the equivalence in (a) since $\beta_{0,r+2}^{S_0}(I)\leq 1$ due to Proposition \ref{35}(b). 

It remains to show that $S_0/I_{\langle r+1\rangle}$ is Cohen-Macaulay provided one of the conditions (ii), (iii) in (a). To see this, we use \cite[Lemma 1.2]{HH99} and Proposition \ref{35}(b),(c) to obtain the vanishing of certain graded Betti numbers, namely
$$\beta_{i,r}(S_0/I_{\langle r+1\rangle})=\beta_{i,r}(S_0/I)=0,$$
for all $i>e$. Then the linearity of $I_{\langle r+1\rangle}$ induces the estimate $\mathrm{proj.dim}_{S_0}(S_0/I_{\langle r+1\rangle})\leq e$. Using the Auslander-Buchsbaum formula we obtain $\depth(S_0/I_{\langle r+1\rangle})=n+1$ and therefore $S_0/I_{\langle r+1\rangle}$ is Cohen-Macaulay.
\end{proof}

Now we are ready to state the main theorem.

\begin{theorem}[Componentwise linearity]\label{45}
Let $X\subset \bP^{n+e}$ be a non-degenerate closed subscheme of codimension $e$ and reduction number $r$. Let $I_X\subset S_0$ be the saturated defining ideal of $X$ and $R_X=S_0/I_X$. Suppose $\deg(X)=\binom{e+r}{e}-1$ and $\depth(R_X)=n$. The following statements hold true:
\begin{itemize}
\item[(a)] $I_X$ is componentwise linear if and only if either
\begin{itemize}
\item[(i)] $\reg(R_X)=r$, or
\item[(ii)] $\reg(R_X)=r+1$ and $\beta_{1,r+1}^{S_0}(R_X)=1$, or
\item[(iii)] $\reg(R_X)\geq r+2$.
\end{itemize}
\item[(b)] $I_X$ is not componentwise linear if and only if $$\reg(R_X)=r+1\ \text{  and } \ \beta_{1, r+1}^{S_0}(R_X)=0.$$
\end{itemize}
\end{theorem}
\begin{proof}
(a) If $I_X$ is componentwise linear then $(I_X)_{\langle r+1\rangle}$ has a linear resolution, the necessary condition hence follows from Proposition \ref{43}(a). 

Conversely, assume either $(i)$, $(ii)$, or $(iii)$. If $\reg(R_X)=r$ then $I_X=(I_X)_{\langle r+1\rangle}$ has a linear resolution due to Proposition \ref{35}(a), therefore $I_X$ is componentwise linear. If $\reg(R_X)>r$ then $(I_X)_{\langle r+1\rangle}$ has a linear resolution  by Proposition \ref{43}. Furthermore, note that $I_X$ has a set of generators consisting of a form of degree $\reg(R_X)+1$ and forms of degree $r+1$ (cf. Theorem \ref{31}). Hence for $1\leq s< \reg(I_X)-r$, we have $(I_X)_{\langle r+s\rangle}=\fm^{s-1}(I_X)_{\langle r+1\rangle}$ which has a linear resolution by Lemma \ref{41}. For $s\geq \reg(I_X)-r$, the ideal $(I_X)_{\langle r+s\rangle}$ has a linear resolution (see \cite[Proposition 1.1]{EG84}).

\medskip
\noindent(b) We have
$$\beta_{1, r+1}^{S_0}(R_X)\leq \beta_{1,r+1}^{S_0}(S_0/\ini(I_X))=1.$$
The latter equality follows from Proposition \ref{35}(b) and Corollary \ref{42}. Hence if $\beta_{1,r+1}^{S_0}(R_X)\not=1$, then $\beta_{1,r+1}^{S_0}(R_X)=0$. Now $(b)$ follows immediately from $(a)$.
\end{proof}

\begin{remark}\label{46} Let $X$ be a non-ACM projective variety of almost maximal degree.

\noindent (a) It is already shown in \cite[Theorem 5.1]{CK18} that $X$ is of almost maximal arithmetic depth.

\smallskip
\noindent (b) Suppose that $X$ satisfies either $(ii)$ or $(iii)$ in Theorem \ref{45}(a), there is an embedding $X\subset Y\subset \bP^{n+e}$ where $Y$ is an $(r+1)$-linear ACM  projective subscheme of the same dimension as $X$.

Indeed, since $I_X$ is componentwise linear, $(I_X)_{\langle r+1\rangle}$ is $(r+1)$-linear and $S_0/(I_X)_{\langle r+1\rangle}$ is Cohen-Macaulay by Proposition \ref{43}(a). Let $Y=\mathrm{Proj}(S_0/(I_X)_{\langle r+1\rangle})$ then the conclusion follows.
\end{remark}

We have seen in Proposition \ref{35} an explicit description of the Betti table of a projective subscheme of almost maximal degree and almost maximal arithmetic depth. In the case $\reg(R_X)\not=r+1$, each graded Betti number is computed precisely. It is natural to ask for the case $\reg(R_X)=r+1$. As an application of the componentwise linearity, we give in the next corollary a partial answer to this question.

\begin{corollary}\label{44}
Let $X\subset \bP^{e+n}$ be a projective subscheme as in Theorem \ref{45}, in particular $\deg(X)=\binom{e+r}{e}-1$ and $\depth(R_X)=n$. We assume further that $\reg(R_X)=r+1$ and $X$ has a componentwise linear resolution. Then the Betti table of the homogeneous coordinate ring of $X$ is
\begin{figure}[!htb]
\begin{tabular}{>{\centering}m{1.5cm}|>{\centering}m{1cm} >{\centering}m{1cm}>{\centering}m{1cm} >{\centering}m{1cm} >{\centering}m{1cm} >{\centering}m{1cm} c}
	&	$0$	&		$1$&	$\ldots$ &	$i$&	$\ldots$ &$e$&	$e+1$\\
\hline
$0$	&	$1$	&		--&		$\ldots$	&--&		$\ldots$	&	--&		--\\
$r$	&	--&	 	$\beta_{1r}$		&	$\ldots$	&$\beta_{ir}$		&	$\ldots$	&	$\beta_{e,r}$&		--\\
$r+1$	&	--&	 	$\binom{e}{0}$		&	$\ldots$	&$\binom{e}{i-1}$		&	$\ldots$	&	$\binom{e}{e-1}$&		$\binom{e}{e}$\\
\end{tabular}
\end{figure}

\noindent where for $1\leq i\leq e+1$,
$$\beta_{i,r}=\binom{e+r}{i+r}\binom{i+r-1}{r},$$
$$\beta_{i,r+1}=\binom{e}{i-1}.$$
\end{corollary}
\begin{proof}
We have seen in Proposition \ref{35}(b) that the Betti table of $R_X$ has the shape as above, where
$$\beta_{i, r}-\beta_{i-1,r+1}=\binom{e+r}{i+r}\binom{r+i-1}{r}-\binom{e}{i-2}.$$

On the other hand, due to Proposition \ref{43}, the ring $S_0/I_{\langle r+1\rangle}$ is Cohen-Macaulay and has an $r$-linear resolution. Hence $S_0/I_{\langle r+1\rangle}$ has a maximal degree by \cite[Theorem 1.1]{CK18}. Now we use again \cite[Lemma 1.2]{HH99} together with \cite[Corollary 1.2]{CK18} to obtain
$$\beta_{ir}=\beta_{ir}(S_0/(I_X)_{\langle r+1\rangle})=\binom{e+r}{i+r}\binom{i+r-1}{r},$$
for all $1\leq i\leq e$. This proves the corollary.
\end{proof}

\begin{example}\label{47}
Let $C$ be the smooth rational curve in $\bP^3$ defined by $(s, t)
\mapsto (s^5, s^4t+s^3t^2, st^4, t^5)$. The curve $C$ has reduction number $r=2$, Castelnuovo-Mumford regularity $\reg(C)=4$ and degree $\deg(C)=5=\binom{2+2}{2}-1$. In particular, $C$ is of almost maximal degree. The Betti table of $C$ is

\begin{figure}[!htb]
\begin{tabular}{>{\centering}m{2cm}|>{\centering}m{1cm} >{\centering}m{1cm} >{\centering}m{1cm} c}
			&	0&		1&	2&	3\\
			\hline
		0	&	1&		--&	--&	--\\
		1	&	--&		--&	--&  --\\
		2	&	--&		4&	3&  --\\
		3	&	--&		1&	2&  1
\end{tabular}
\end{figure}
\noindent By Theorem \ref{45}(a), $C$ has a componentwise linear resolution. This can be also shown by direct computation. Indeed, using Macaulay 2 we can find the Betti table of $S_0/I_{\langle 3\rangle}$

\begin{figure}[!htb]
\begin{tabular}{>{\centering}m{2cm}|>{\centering}m{1cm} >{\centering}m{1cm} c}
			&	0&		1&	2\\
			\hline
		0	&	1&		--&	--\\
		1	&	--&		--&	--\\
		2	&	--&		4&	3
\end{tabular}
\end{figure}

\newpage
\noindent In particular, $S_0/I_{\langle 3\rangle}$ is Cohen-Macaulay. Furthermore, $S_0/I_{\langle 4\rangle}$ has depth zero and its Betti table is

\begin{figure}[!htb]
\begin{tabular}
{>{\centering}m{2cm}|>{\centering}m{1cm} >{\centering}m{1cm} >{\centering}m{1cm} >{\centering}m{1cm} c}
			&	0&		1&	2&	3 & 4\\
			\hline
		0	&	1&		--&	--&	--& --\\
		1	&	--&		--&	--&  --& --\\
		2	&	--&		--&	--&  --& --\\
		3	&	--&		14&	26&17&4
\end{tabular}
\end{figure}
\end{example}

In \cite[Examples 5.5, 5.6, 5.7]{CK18}, it is shown that all cases (i), (ii), (iii) in part (a) of Theorem \ref{45} actually occur. In the next example, we will see a rational curve in $\bP^3$ satisfying all conditions in part (b) of Theorem \ref{45}.

\begin{example}\label{48}
Let $C$ be a smooth rational curve in $\mathbb P^3$ defined by $(s,t)\mapsto (s^9, s^4t^5+s^5t^4, s^4t^5+s^7t^2, t^9)$. Let $I_C\subset S_0=k[x,y,z,w]$ be its defining ideal and $R_C=S_0/I_C$. We have $R/(x,w)\simeq S_0/(x, w)+(y,z)^4$, so the reduction number of $R$ is $r=3$. The curve $C$ is of almost maximal degree with $\deg(C)=9=\binom{2+3}{2}-1$.

On the other hand, the Betti table of $R$ is

\begin{figure}[!htb]
\begin{tabular}{>{\centering}m{2cm}|>{\centering}m{1cm} >{\centering}m{1cm} >{\centering}m{1cm} c}
			&	0&		1&	2&	3\\
			\hline
		0	&	1&		--&	--&	--\\
		1	&	--&		--&	--&  --\\
		2	&	--&		--&	--&  --\\
		3	&	--&		5&	3&  --\\
		4	&	--&		--&	2&  1
\end{tabular}
\end{figure}

The curve $C$ satisfies all conditions in part (b) of Theorem \ref{45} and its minimal free resolution is not componentwise linear.
\end{example}

If the codimension and reduction number are fixed, the size of the Betti tables of projective varieties of almost maximal degree are bounded. While the projective dimension equals to the codimension plus one, the Casteluovo-Mumford regularity has the following bound.

\begin{proposition}\label{49}
Let $X\subset\bP^{e+n}$ be a projective variety of codimension $e$ and reduction number $r$. Suppose $X$ is of almost maximal degree. Then
$$\reg(X)\leq \deg(X)-e+1.$$
\end{proposition}
\begin{proof}
Let $I_X\subset S_0=k[x_0, \ldots, x_{e+n}]$ be the saturated defining ideal of $X$ and $R_X=S_0/I_X$ be the homogeneous coordinate ring. Then $R_X$ has depth $\geq n$ by \cite[Theorem 5.1]{CK18}. 

Since $k$ is infinite, we apply the Bertini irreducibility theorem (see \cite[Theorem 6.3(4)]{Jou83}) to choose appropriate linear forms $x_e, \ldots, x_{e+n}$ such that 
\begin{enumerate}[(a)]
\item $x_{e+1}, \ldots, x_{e+n}$ is a regular sequence on $R_X$;
\item $S=k[x_e, \ldots, x_{e+n}] \rightarrow R$ is a Noether normalization with reduction number $r_S(R)=r$;
\item $R_X/(x_{e+2}, \ldots, x_{e+n})$ is a domain.
\end{enumerate}

We have $\reg(R_X)=\reg(R_X/(x_{e+2}, \ldots, x_{e+n}))$, $r=r(R_X/(x_{e+2}, \ldots, x_{e+n}))$ and
$$\deg(R_X)=\deg(R_X/(x_{e+2}, \ldots, x_{e+n})).$$
Now a famous result of L. Gruson, R. Lazarsfeld and C. Peskine \cite[Theorem 1.1]{GLP83} shows that 
$$\reg(R_X/(x_{e+2}, \ldots, x_{e+n}))\leq \deg(R_X/(x_{e+2}, \ldots, x_{e+n}))-e=\deg(R_X)-e.$$
Therefore,
$$\reg(X)\leq \deg(R_X)-e+1=\binom{e+r}{e}-e.$$
\end{proof}

For a projective variety $Y\subset \bP^{e+n}$ we have $\reg(Y)\geq 2$. The equality occurs if and only if $r(Y)=1$, if and only if $Y$ has minimal degree (see \cite[Corollary 3.5]{CK18}). Projective varieties of almost maximal degree as in Proposition \ref{49} do not have minimal degree, so we have the inequalities
$$3\leq r+1\leq \reg(X)\leq \binom{e+r}{e}-e.$$
These lower and upper bounds for the regularity of an almost maximal degree variety are sharp. Indeed, we have seen in Examples 5.4, 5.6 of \cite{CK18} a smooth elliptic curve $C$ and a smooth rational curve $C^\prime$ in $\bP^3$ both of reduction number $2$ and of  almost maximal degree $5$. The Casteluovo-Mumford regularity of each curve is
$$\reg(C)=3=r+1, \ \reg(C^\prime)=4=\binom{e+r}{e}-e.$$

%%%%%%%%%%%%%%%%%%%%%%%%%%%%%%%


\begin{thebibliography}{99}

\bibitem{AHH00} A. Aramova, J. Herzog, and T. Hibi, Ideals with stable Betti numbers,  Adv. Math. {\bf 152}(1), 72-77 (2000).

\bibitem{BG01} I. Bermejo and Ph. Gimenez, Computing the Castelnuovo-Mumford regularity of some subschemes of $P^n_K$ using quotients of monomial ideals. In: {\it Effective Methods in Algebraic Geometry} (Bath, 2000), J. Pure Appl. Algebra {\bf 164}(1-2), 23-33 (2001).

%\bibitem{BGGM17} I. Bermejo, E. Garcia-Llorente, I. Garcia-Marco, M. Morales, Noether resolutions in dimension $2$,  J. Algebra} {\bf 482} (2017) 398-426.

\bibitem{CK18} {\fontencoding{T5}\selectfont \DJ. T. C\uhorn{}\`\ohorn ng} and S. Kwak, The reduction number and degree bound of projective subschemes, Trans. Amer. Math. Soc. {\bf 373}(2), 1153-1180 (2020).

%\bibitem{Eis} D. Eisenbud, Geometry of syzgyies. %{\it GTM} {\bf 150}. Springer-Verlag, New York 1995.

\bibitem{EG84} D. Eisenbud and S. Goto, Linear free resolutions and minimal multiplicity,  J. Algebra {\bf 88}, 89-133 (1984).

%\bibitem{EGHP06} D. Eisenbud, M. Green, K. Hulek, S. Popescu, Small schemes and varieties of minimal degree, {\it Am. J. Math.} {\bf 128}(6) (2006) 1363-1389.

%\bibitem{EGHP05} D. Eisenbud, M. Green, K. Hulek, S. Popescu, Restricting linear syzygies: algebra and geometry,  Compos. Math.} {\bf 141} (6) (2005) 1460-1478.
%

%\bibitem{FHT09} C. A. Francisco, H. T. H\`a, A. Van Tuyl, Splittings of monomial ideals,  Proc. Amer. Math. Soc.} {\bf 137} (2009), 3271-3282.

\bibitem{FMN06} C. A. Francisco, J. Migliore, and U. Nagel, On the componentwise linearity and the minimal free resolution of a tetrahedral curve,  J. Algebra {\bf  299}, 535-569 (2006).

\bibitem{GS} D. R. Grayson and M. E. Stillman, Macaulay2 - a software system for research in algebraic geometry. Available at \url{http://www.math.uiuc.edu/Macaulay2/}

\bibitem{Green98} M. Green, Generic initial ideals. In: J. Elias, J.M. Giral, R.M. Mir\'o-Roig, S. Zarzuela (Eds.), Six Lectures on Commutative Algebra, {\it Progress in Mathematics} {\bf 166}, Birkh{\"a}user, pp. 119-186 (1998).

\bibitem{GLP83} L. Gruson, R. Lazarsfeld, and C. Peskine, On a theorem of Castelnuovo, and the equations defining space curves, Invent. Math. {\bf 72}, 491-506 (1983).

\bibitem{HH99} J. Herzog and T. Hibi, Componentwise linear ideals,  Nagoya Math. J. {\bf 153}, 141-153 (1999).

\bibitem{HH11} J. Herzog and T. Hibi, Monomial ideals, {\it Graduate Texts in Mathematics} {\bf 260}, Springer-Verlag London 2011.

\bibitem{HHMT08} J. Herzog, T. Hibi, S. Murai, and Y. Takayama, Componentwise linear ideals with minimal or maximal Betti numbers, Ark. Mat. {\bf 46}(1), 69-75 (2008).

\bibitem{Jou83} J.-P. Jouanolou, Th\'eor\`emes de Bertini et applications,  {\it Progress in Mathematics} {\bf 42}, Birkh\"auser Boston Inc., Boston, MA, 1983.

\bibitem{Nagel03} U. Nagel, Non-degenerate curves with maximal Hartshorne-Rao module,  Math. Z. {\bf 244}(4), 753-773 (2003).

%\bibitem{Nagel05} U. Nagel, Comparing Castelnuovo-Mumford regularity and extended degree: the borderline cases,  Trans. Amer. Math. Soc. {\bf 357} (2005), 3585-3603.

\bibitem{NR15} U. Nagel and T. R\"omer, Criteria for componentwise linearity,  Communications in Algebra {\bf 43}(3), 935-952 (2015).

%\bibitem{Naito02} H. Naito, Minimal free resolution of curves of degree 6 or lower in the 3-dimensional projective space,  Tokyo J. Math.} {\bf 25}(1) (2002), 191-196.

\bibitem{NVT87} N. V. Trung, Reduction exponent and degree bound for the defining equations of graded rings,  Proc. Amer. Math. Soc. {\bf 101}, 229-236 (1987).

\bibitem{Vas96} W. V. Vasconcelos, The reduction number of an algebra,  Compositio Math. {\bf 104}, 189-197 (1996).

\end{thebibliography}
\end{document}